\documentclass[paper]{siamart190516}

\usepackage{lipsum}
\usepackage{amsfonts}
\usepackage{graphicx}
\usepackage{epstopdf}
\usepackage{algorithmic}
\usepackage{enumerate}

\ifpdf
  \DeclareGraphicsExtensions{.eps,.pdf,.png,.jpg}
\else
  \DeclareGraphicsExtensions{.eps}
\fi

\newsiamremark{remark}{Remark}
\newsiamremark{hypothesis}{Hypothesis}
\crefname{hypothesis}{Hypothesis}{Hypotheses}
\newsiamthm{claim}{Claim}
\newsiamthm{example}{Example}
\headers{Error bounds for the absolute value equations}{M. Zamani, M. Hlad\'{i}k}

\title{Error bounds and a condition number for the absolute value equations \thanks{Submitted to the editors DATE.
}}

\author{Moslem Zamani\thanks{Department of Applied Mathematics, Ferdowsi University of Mashhad, Iran 
  (\email{moslem.zamani@um.ac.ir}).}
\and Milan Hlad\'{i}k\thanks{Charles University, Faculty  of  Mathematics  and  Physics,
Department of Applied Mathematics, 
Malostransk\'e n\'am.~25, 11800, Prague, Czech Republic,
  (\email{hladik@kam.mff.cuni.cz}).
  }}
\usepackage{amsopn}

\ifpdf
\hypersetup{
  pdftitle={Error bounds and a condition number for the absolute value equations},
  pdfauthor={M Zamani, and J. E. Smith}
}
\fi

\begin{document}

\maketitle

% REQUIRED
\begin{abstract}
Absolute value equations, due to their relation to the linear complementarity problem, have been intensively studied recently. In this paper, we present error bounds for absolute value equations. Along with the error bounds, we introduce an appropriate condition number. We consider general scaled matrix p-norms, as well as particular p-norms. We discuss basic properties of the condition number, its computational complexity, its bounds and also exact values for special classes of matrices. We consider also matrices that appear based on the transformation from the linear complementarity problem.
\end{abstract}

% REQUIRED

% REQUIRED
%%%%%%%%%%%%%%%%%%%%%%%%%%%%%
\section{Introduction}
We consider the absolute value equation problem of finding an $x\in\mathbb{R}^n$ such that
\begin{equation}\tag{AVE}\label{AVE}
Ax-b=|x|,
\end{equation}
where $A\in\mathbb{R}^{n\times n}$, $b\in\mathbb{R}^{n}$ and $|\cdot|$ denotes absolute value. A slightly more generalized form of  \eqref{AVE} was introduced by Rohn \cite{Rohn1}, which is written as 
\begin{equation*}
Ax-b=B|x|,
\end{equation*}
where $B\in\mathbb{R}^{n\times n}$, but we will deal merely with \eqref{AVE}. 

Many methods, including Newton-like methods \cite{Cruz, MangasarianN, Zhang} or concave optimization methods \cite{MangasarianC1, MangasarianC2}, have been developed for  \eqref{AVE}. The important point concerning numerical methods is the precision of the computed solution. To the best knowledge of the authors, there exist only few papers which are devoted to this subject for \eqref{AVE}; for instance see \cite{abda, wang2, wang1}. Wang et al. \cite{wang2, wang1} use interval methods for numerical validation. In addition, some general bounds for the solution set were presented in~\cite{Hla2018b}. One effective method  for numerical validation is error bound method \cite{Pang}. 

Error bounds play a crucial role in theoretical and numerical analysis of linear algebraic and optimization problems \cite{coulibaly, fabian, facchinei2007finite, higham, Pang}.  In this paper, we study error bounds for \eqref{AVE}. 
%Indeed, for regular matrix 
%$[A-I, A+I]$ and a scaling $p$-norm, $\|. \|$, on $\mathbb{R}^n$, we present an error bound 
%\begin{align*}
%\|x- x^\star\|\leq \max_{\|d\|_{\infty}\leq 1} \|(A-\diag(d))^{-1}\| \|Ax-b-|x|\|, 
%\begin{align*}
%where $x^\star$ denote the unique solution of \eqref{AVE} and matrix norm is regarded as the norm induced by $\| .\|$. We compute the given error bound for some classes of matrices and norms. 
Indeed, under the assumption guaranteeing unique solvability for each $b\in\mathbb{R}^n$, we compute upper bounds for $\|x- x^\star\|$, the distance to the solution $x^\star$, in terms of a computable residual function. We discuss various kinds of norms and investigate special classes of matrices. 

It is well-known that a linear complementarity problem can be formulated as an absolute value equation \cite{mangasarian2}. In fact, it is one of the main applications of  absolute value equations. In Section~\ref{sLCP}, we study error bounds for absolute value equations obtained by the reformulation of linear complementarity problems. In addition, thanks to the given results, we provide a new error bound for linear complementarity problems. 

The paper is organized as follows. After reviewing terminologies and notations, we investigate error bounds for absolute value equations in Section~\ref{sAVE}. Section~\ref{sLCP} is devoted to linear complementarity problems. We study relative condition number of AVE in Section \ref{S.RC}.

\subsection{Notation}
The $n$-dimensional Euclidean space is denoted by $\mathbb{R}^n$.  Vectors are considered to be column vectors and the superscript $T$ represents the transpose operation. We use $e$ and $I$ to denote the vector of ones and the identity matrix, respectively. We denote an arbitrary scaling $p$-norm on $\mathbb{R}^n$ by $\|\cdot\|$, that is, $\| x\|=\| Dx\|_p$ for a positive diagonal matrix $D$ and a $p$-norm. In particular, $\|\cdot\|_1$, $\|\cdot\|_2$ and $\|\cdot\|_\infty$ stand for 
$1$-norm, $2$-norm and $\infty$-norm, respectively. We use $\sgn(x)$ to denote the sign of $x$.

Let $A$ and $B$ be $n\times n$ matrices. We denote the smallest singular value and the spectral radius of $A$ by $\sigma_{\min}(A)$ and $\rho(A)$, respectively. We use 
$\lambda(A)$ to denote the vector of eigenvalues of a symmetric matrix $A$, and $\lambda_{\min}(A)$ and $\lambda_{\max}(A)$ stand for the smallest and the largest eigenvalue, respectively. For a given norm $\|\cdot\|$ on $\mathbb{R}^n$, $\| A\|$ denotes the matrix norm induced by $\|\cdot\|$, which is defined as 
\begin{align*}
\| A\|=\max\{\|Ax\|: \|x\|=1\}.
\end{align*}
The matrix inequality $A\geq B$, $| A|$ and $\max(A, B)$ are understood entrywise. For $d\in\mathbb{R}^n$, $\diag(d)$ stands for the diagonal matrix  whose entries on the diagonal are the components of $d$. In contrast, $\Diag(A)$ denotes the vector of diagonal elements of  $A$. The $i$th row and $i$th column of $A$ are denoted by $A_{i\ast}$ and $A_{\ast i}$, receptively. 
We denote the comparison matrix $A$ by $\langle A\rangle$, which is defined as
 \begin{align*}
 &  \langle A\rangle _{ii}=|A_{ii}|,  && i=1, ...,n,\\
   & \langle A\rangle _{ij}=-|A_{ij}|,  && i,j=1, ...,n, i\neq j.
 \end{align*}

We recall the following definitions for an $n\times n$ real matrix $A$:
\begin{itemize}
\item
$A$ is a P-matrix if each principal minor of $A$ is positive.
\item
$A$ is an M-matrix if $A^{-1}\geq 0$ and $A_{ij}\leq 0$ for $i, j = 1,2,...,n$ with $i\neq j$.
\item
$A$ is an H-matrix if its comparison matrix is an M-matrix.
\end{itemize}

We will exploit some  results from interval linear algebra, so we recall some results from this discipline. For two $n\times n$  matrices $\unum{A}$ and $\onum{A}$, $\unum{A}\leq\onum{A}$, the interval matrix 
$\bold{A} = [\unum{A}, \onum{A}]$ is defined as $\bold{A}=\{A: \unum{A}\leq A \leq \onum{A}\}$.  
An interval matrix A is called regular if each $A\in\bold{A}$ is nonsingular. Furthermore, we denote and define the inverse of a regular interval matrix $\bold{A}$ as
 $\bold{A}^{-1}:=\{A^{-1}: A\in \bold{A}\}$. Note that the inverse of an interval matrix is not necessarily an interval matrix.

In this paper, generalized Jacobian matrices~\cite{Clarke} are used in the presence of nonsmooth functions. Let $f\colon \R^n \to\R^m$ be a locally Lipschitz function. The generalized gradient of $f$ at $\hat x$, denoted by $\partial f(\hat x)$, is defined as 
\begin{align*}
\partial f(\hat x):=
 \co\{\textstyle\lim_{n\to \infty} \nabla f(x_n): x_n\to \hat x, x_n\notin  X_f\},
\end{align*}
where $ X_f$ is the set of points at which $f$ is not differentiable and $\co(S)$ denotes the convex hull of a set~$S$.

%%%%%%%%%%%%%%%%%%%%%%%%%%%%%%%%%%%%%%%%%%%%%
\section{Error bounds for the absolute value equations}\label{sAVE}

Consider an absolute value equation \eqref{AVE}.  It is known that \eqref{AVE} has a unique solution for each $b\in\mathbb{R}^n$ if and only if the interval matrix $[A-I, A+I]$ is regular; see Theorem~3.3 in \cite{Wu}. 
That is why in many statements below, we make an assumption that the interval matrix $[A-I, A+I]$ is regular.
In this case, we denote the unique solution set of  \eqref{AVE} by $x^\star$. 
%In the next theorem, we present an error bound for \eqref{AVE} when $[A-I, A+I]$ is regular.

\begin{theorem}\label{TB}
%Let $\| .\|$ be a norm on $\mathbb{R}^n$. 
If the interval matrix $[A-I, A+I]$ is regular, then 
\begin{align}\label{Er}
\|x- x^\star\| \leq \max_{\|d\|_{\infty}\leq 1} \|(A-\diag(d))^{-1}\|\cdot \|Ax-b-|x|\|, \ \  \forall x\in \mathbb{R}^n.
\end{align}
\end{theorem} 

\begin{proof}
Note that due to regularity of $[A-I, A+I]$ the right side of the above inequality is finite. Define the residual function
$\phi\colon \R^n\to \R^n$ by $\phi(x)=Ax-b-|x|$. By the mean value theorem, see Theorem~8 in~\cite{Hirr},
\begin{align*}
\phi(x)=\phi(x)-\phi(x^\star)=(\textstyle\sum_{i=1}^n \lambda_i A_i)(x- x^\star),
\end{align*}
where $A_i\in \partial  \phi(x_i)$, $x_i\in co(\{x, x^\star\})$, $\lambda_i\geq 0$, $i=1, ..., n$ and $\sum_{i=1}^n \lambda_i=1$. It is easily seen that  $\partial\phi(y)\subseteq \{A+\diag(d): \|d\|_{\infty}\leq 1\}=[A-I, A+I]$ for $y\in\mathbb{R}^n$. Due to the convexity of $\{A+\diag(d): \|d\|_{\infty}\leq 1\}$, we have 
\begin{align*}
\phi(x)=\hat A(x- x^\star),
\end{align*}
for some $\hat A\in [A-I, A+I]$. By multiplying $\hat A^{-1}$ on the both sides and using induced norms property, we obtain
\begin{align*}
\|x- x^\star\|=\|\hat A^{-1}\phi(x)\|\leq \|\hat A^{-1}\| \|\phi(x)\|\leq \max_{\|d\|_{\infty}\leq 1} \|(A-\diag(d))^{-1}\|\|\phi(x)\|,
\end{align*}
which completes the proof.
\end{proof}

%In the rest, we assume that $[A-I, A+I]$ is regular.
To take advantage of this formulation, we need to compute the optimal value of the following optimization problem,
\begin{align}\label{ONEr}
c(A):=\max \ & \|(A-\diag(d))^{-1}\| \ \
s.t. \ \ \|d\|_{\infty}\leq 1.
\end{align}
We call the optimal value of \eqref{ONEr} \emph{the condition number} of the absolute value equation~\eqref{AVE} with respect to the norm $\|\cdot\|$. In addition, we denote \emph{the condition number} with respect to the $1$-norm, $2$-norm and $\infty$-norm by $c_1(A)$, $c_2(A)$ and $c_\infty(A)$, respectively.
By properties of matrix norms, we have the following results.

\begin{proposition}
Let $[A-I, A+I]$ be regular and $\alpha$ be a scalar with $|\alpha|\geq1$. Then, 
\begin{enumerate}[i)]
\item
$c(-A)=c(A)$;
\item
$c_1(A^T)=c_\infty(A)$;
\item
$c(\alpha A)\leq |\alpha^{-1}| c(A)$.
\end{enumerate}
\end{proposition}

\begin{proof}
Part i) and ii) are straightforward. Part iii) follows form the fact that
\begin{align*}
\max_{\|d\|_{\infty}\leq 1} \  \|(\alpha A-\diag(d))^{-1}\|&= \max_{\|d\|_{\infty}\leq |\alpha^{-1}|} \  \|(\alpha A-\alpha \diag(d))^{-1}\|\\
& \leq |\alpha^{-1}|\max_{\|d\|_{\infty}\leq 1} \  \|( A-\diag(d))^{-1}\|.
\end{align*}
\end{proof}

In the next proposition, we show that optimization problem \eqref{ONEr} attains its minimum at some vertices of the box $\{d:\|d\|_\infty\leq 1\}=[-e,e]$. 

\begin{proposition}\label{Pr1}
Let  the interval matrix $[A-I, A+I]$ be regular. Then, there exists a vertex of polytope $\{d:\|d\|_\infty\leq 1\}$ which is  a solution of \eqref{ONEr}.
\end{proposition}

\begin{proof}
We will show that problem  \eqref{ONEr} has a solution whose components are  either one or minus one. As the feasible set is compact, problem  \eqref{ONEr} attains its maximum. Let $\hat d$ be an optimal solution. If $\hat d$ is a vertex of $\{d:\|d\|_\infty\leq 1\}$, the proof will be complete. Otherwise, without loss of generality, suppose that 
$|\hat d_1|<1$. Let $f\colon  [-1, 1]\to \mathbb{R}$ given by 
$f(t)=\|(A-\diag((t, \check d)))^{-1}\|$, where $\check d$ is obtained by removing the first component of
 $\hat d$. By Sherman-Morrison formula \cite{Horn}, $f(t)=\|\hat A^{-1}-\frac{t}{1+t\hat A^{-1}_{11}}E\|$, where $\hat A=A+\diag((0, \check d))$ and
  $E=\hat A^{-1}_{\ast1}\hat A^{-1}_{1\ast}$. Due to regularity of
 $[A-I, A+I]$,  $\frac{t}{1+t\hat A^{-1}_{11}}$ is well-defined for $t\in[-1, 1]$. Consider the optimization problem 
 $\max_{t\in[-1, 1]} f(t)$. Since $\|A+tE\|$ as a function of $t$ is convex and $g(t)=\frac{t}{1+t\hat A^{-1}_{11}}$  is strictly monotone on $[-1, 1]$, $f$ is convex on its domain \cite{Boyde}, and consequently $\max_{t\in[-1, 1]} f(t)=\max\{{f(-1), f(1)}\}$. Hence, due to optimality of $\hat d$, we get a new point $\tilde d$ which is optimal to \eqref{ONEr} and all components instead of first one are equal to $\hat d$ and its first component is either one or minus one. In the same line, one can obtain a solution $\tilde d$ with $|\tilde d|=e$, which completes the proof. 
\end{proof} 

By Proposition \ref{Pr1}, to handle problem \eqref{ONEr}, one needs to check solely all vertices of  $\{d:\|d\|_\infty\leq 1\}$. As the number of vertices is $2^n$, this method may not be effective for large $n$. 
Indeed, problem \eqref{ONEr} is NP-hard in general. It is known that for any rational $p\in [1, \infty)$, except for $p=1, 2$, computation of the matrix $p$-norm of a given matrix is NP-hard \cite{Hend}.  Consequently, problem \eqref{ONEr} is NP-hard for any rational $p\in [1, \infty)$ except $p=1, 2$. We prove intractability for 1-norm, so it is NP-hard for $\infty$-norm, too.  We conjecture it is also NP-hard for 2-norm. 

\begin{lemma}\label{lmmNormBoundPropNP}.
For any $u,v\in\R^n$ and any vector norm we have either $\|u\|\leq\|u+v\|$ or $\|u\|\leq\|u-v\|$.
\end{lemma}

\begin{proof}
By triangle inequality $\|u\|=\frac{1}{2}\|u+v+u-v\|\leq\frac{1}{2}(\|u+v\|+\|u-v\|)$, from which the statement follows.
\end{proof} 

\begin{proposition}
Computation of $c_1(A)$ is an NP-hard problem.
\end{proposition}

\begin{proof}
By \cite{Roh2000}, solving the problem
\begin{align}\label{maxNpRohn}
\max\ e^T|x|  \st |Ax|\leq e
\end{align}
is NP-hard. Even more, it is intractable even with accuracy less than $\frac{1}{2}$ when $A^{-1}$ is a so called MC-matrix~\cite{Roh2000}.
Recall that $M\in\R^{n\times n}$ is an MC matrix if it is symmetric, $M_{ii}=n$ and $M_{ij}\in\{0,-1\}$, $i\not=j$. For an MC matrix $M$ we have $\lambda_{\max}(M)\leq 2n-1$, from which $\lambda_{\min}(M^{-1})\geq \frac{1}{2n-1}$. Therefore $\lambda_{\min}(A)\geq \frac{1}{2n-1}$ and we can achieve $\lambda_{\min}(A)>1$ by a suitable scaling. As a consequence, $[A-I,A+I]$ is regular. 

Feasible solutions to the above optimization problem can be equivalently characterized as
\begin{align*}
Ax=b,\ \ b\in[-e,e],
\end{align*}
or, substituting $b=\diag(b)e=\diag(b)y$ with $y=e$, 
\begin{align*}
\begin{pmatrix}
 A & -\diag(b) \\ 0 & I
\end{pmatrix}
\begin{pmatrix}
 x \\ y
\end{pmatrix}
=\begin{pmatrix}
 0 \\ e
\end{pmatrix},\ \ b\in[-e,e].
\end{align*}
Introducing an auxiliary variable $z=1$, we get
\begin{align*}
\begin{pmatrix}
 A & -\diag(b) & 0 \\ 0 & I & -e \\ 0 & 0 & 1
\end{pmatrix}
\begin{pmatrix}
 x \\ y \\ z
\end{pmatrix}
=\begin{pmatrix}
 0 \\ 0 \\ 1
\end{pmatrix},\ \ b\in[-e,e].
\end{align*}
Rewrite the system as
\begin{align*}
\begin{pmatrix}
 D & A & 0 \\ I & 0 & -e \\ 0 & 0 & 1
\end{pmatrix}
\begin{pmatrix}
 y \\ x \\ z
\end{pmatrix}
=\begin{pmatrix}
 0 \\ 0 \\ 1
\end{pmatrix},\ \ |D|\leq I.
\end{align*}
Let $\alpha>0$ be sufficiently large. The system equivalently reads
\begin{align*}
\begin{pmatrix}
 D & A & 0 \\ \alpha I & 0 & -e\alpha \\ 0 & 0 & 2
\end{pmatrix}
\begin{pmatrix}
 \frac{1}{\alpha}y \\[2pt] \frac{1}{\alpha}x \\[2pt] \frac{1}{\alpha}z
\end{pmatrix}
=\begin{pmatrix}
 0 \\ 0 \\ \frac{2}{\alpha}
\end{pmatrix},\ \ |D|\leq I.
\end{align*}
Now, we relax the system by introducing intervals on the remaining diagonal entries
\begin{align*}
\begin{pmatrix}
 D & A & 0 \\ \alpha I & D' & -e\alpha \\ 0 & 0 & 2+d
\end{pmatrix}
\begin{pmatrix}
 \frac{1}{\alpha}y \\[2pt] \frac{1}{\alpha}x \\[2pt] \frac{1}{\alpha}z
\end{pmatrix}
=\begin{pmatrix}
 0 \\ 0 \\ \frac{2}{\alpha}
\end{pmatrix},\ \ |D|,|D'|\leq I,\ |d|\leq 1.
\end{align*}
Denote by $M(D,D',d)$ the constraint matrix. The solution is $\frac{2}{\alpha}$-multiple of the last column of the inverse matrix $M(D,D',d)^{-1}$. That is why we analytically express the inverse matrix (notice that it exists due to regularity of $[\alpha A-I,\alpha A+I]$)
\begin{align*}\addtolength\arraycolsep{-0.35ex}
M(D,D',d)^{-1}=
\begin{pmatrix}
 -D'C & \frac{1}{\alpha}(I+D'CD) & \frac{1}{2+d}(e+D'CDe) \\[2pt]
 \alpha C & -CD & -\frac{\alpha}{2+d}CDe \\[2pt]
 0 & 0 & \frac{1}{2+d}
\end{pmatrix},\, |D|, |D'|\leq I, |d|\leq 1,
\end{align*}
where $C:=(\alpha A-DD')^{-1}$. 
The idea of the proof is to reduce the above mentioned NP-hard problem to computation of the condition number for matrix $M(0,0,0)$. 
Obviously, 1-norm of $M(D,D',d)^{-1}$ is attained for the value of $d=-1$, so we can fix it for the remainder of the proof. 

\emph{Claim A.} 
There exist $\bar D$ and $\bar D'$ such that $|\bar D|=|\bar D'|=I$ and $c_1(M(0,0,0))=\|M(\bar D, \bar D', -1)^{-1}_{*(2n+1)}\|_1$.

\emph{Proof of the Claim A.}
By Proposition~\ref{Pr1}, the maximum norm is attained for $|D|=|D'|=I$. Therefore, we need only to investigate the matrices with $|D|=|D'|=I$. 
Let $c_1(M(0,0,0))=\|M( D,  D', -1)^{-1}\|_1$ with $|D|=|D'|=I$. If 1-norm of $M( D,  D', -1)^{-1}$ is attained for the last column, the claim is resulted. Otherwise,
since $\alpha>0$ is arbitrarily large, the 1-norm is attained for no column of the middle part.  Suppose that the norm is attained for $i$th column of the first column block. We compare the norms of this column and the last column of $M(D,D',d)^{-1}$, that is, we compare vectors
\begin{align*}
\begin{pmatrix}
-D'C_{*i} \\
\alpha C_{*i} \\
0
\end{pmatrix}\quad\mbox{and}\quad
\begin{pmatrix}
e+D'CDe \\
-\alpha CDe \\
1
\end{pmatrix}.
\end{align*}
We compare separately their three blocks. Obviously, for the last entry the latter is larger. Since 
$C\to 0$ as $\alpha\to\infty$, the first block of entries of the former vector is arbitrarily small and neglectable. Thus we focus on the second block. The former vector has entries $\alpha C_{*i}$. In view of Lemma~\ref{lmmNormBoundPropNP}, one can choose a suitable $\bar D$ such that $|\bar D|=I$ and
 $\|\alpha C_{*i} \|_1\leq \|\alpha C\bar De \|_1=\|\alpha C_{*i}+\alpha\sum_{j\neq i} C_{*j}\bar d_{jj} \|_1$. Furthermore, one can select a matrix $\bar D'$ with $|\bar D'|=I$ and 
  $\|e+D'CDe\|_1=\|e+\bar D'C\bar De\|_1$. Because $c_1(M(0,0,0))=\|M( D,  D', -1)^{-1}\|_1$, the given matrices $\bar D$ and $\bar D'$ fulfill the claim. 

\emph{Claim B.} 
The 1-norm of the last column is arbitrarily close to $1+n+ e^T|A^{-1}De|$. 

\emph{Proof of the Claim B.}
The last entry of the column is~1. Since $C\to 0$ as $\alpha\to\infty$, the first block tends to $e$ as $\alpha\to\infty$.  The second block reads $-\alpha CDe=-(A-\frac{1}{\alpha}DD')^{-1}De$, which tends to $-A^{-1}De$ as $\alpha\to\infty$. So its 1-norm tends to $e^T|A^{-1}De|$.

By Claim B, the 1-norm of the last column is by $1+n$ larger than the objective value of \eqref{maxNpRohn}. So by maximizing 1-norm of $M(D,D',d)^{-1}$ we can deduce the maximum of \eqref{maxNpRohn} with arbitrary precision.
Notice that $e^T|A^{-1}|e$ is an upper bound on \eqref{maxNpRohn} and it has polynomial size, so we can find $\alpha$ of polynomial size, too by the standard means (c.f.~\cite{Schr1998}).
\end{proof}

In general, the computation of $c(A)$ is not easy. However, computation of the condition number with respect to some norms or for some classes of matrices  is not difficult. In the rest of the section, we study the given condition number from this aspect.  

For the following we say that a matrix norm is monotone if $|A|\leq B$ implies $\|A\|\leq\|B\|$. For instance, the scaled matrix p-norms are monotone.

\begin{proposition}\label{propGenUpperBound}
Let $\VERT \cdot \VERT$ be a monotone matrix norm. If $\VERT|A^{-1}|\VERT<1$, then
\begin{align*}
c(A)_{\VERT \cdot \VERT}\leq \frac{\VERT A^{-1}\VERT }{1-\VERT  |A^{-1}| \VERT }.
\end{align*}
\end{proposition}

\begin{proof}
By Proposition \ref{Pr1}, we need to check the vertices of $\{d: \|d\|_\infty\leq 1\}$. Let $d$ be such that $|d|=e$ and denote $D:=\diag(d)$. Then
\begin{align*}
(A-D)^{-1}
=(I-A^{-1}D^{-1})^{-1}A^{-1}.
\end{align*}
Since $\rho(A^{-1}D^{-1})\leq\rho(|A^{-1}|)\leq \VERT|A^{-1}|\VERT<1$, we have by Neumann series~\cite{Horn}
\begin{align*}
(A-D)^{-1}
=\sum_{k=0}^{\infty}\left(A^{-1}D^{-1}\right)^k A^{-1}.
\end{align*}
By monotonicity of the matrix norm
\begin{align*}
\VERT (A-D)^{-1}\VERT 
&\leq\sum_{k=0}^{\infty}\big\VERT A^{-1}D^{-1}\big\VERT ^k \cdot \VERT A^{-1}\VERT \\
&\leq\sum_{k=0}^{\infty}\big\VERT  |A^{-1}| \big\VERT ^k \cdot \VERT A^{-1}\VERT 
 = \frac{\VERT A^{-1}\VERT }{1-\VERT  |A^{-1}| \VERT }.
\end{align*}
\end{proof}

%%%%%%%%%%%%%%%%%%%%%%%%%%%%%%%%%%%%%%%%%%%%%

\begin{theorem}\label{JhJ}
If $\rho(|A^{-1}|)<\gamma<1$, then there exists a scaling $1$-norm $\VERT\cdot\VERT$ such that 
\begin{align}
\VERT x-x^\star \VERT \leq \frac{\gamma}{1-\gamma}\VERT Ax-b-|x|\VERT,
 \ \  \forall x\in \mathbb{R}^n.
\end{align}
\end{theorem}

\begin{proof}
Note that the assumption implies that \eqref{AVE} has a unique solution, see Theorem 4 in \cite{Rohn_f}, and $[A-I, A+I]$ is regular. Due to the continuity of eigenvalues with respect to the matrix elements, there exists matrix $B$ with $|A^{-1}|<B$ and $\rho(B)= \gamma$. By Perron-Frobenius theorem, there exists $v> 0$ such that $Bv=\rho(B)v$. We define norm $\VERT\cdot\VERT$ as $\VERT x\VERT=v^T|x|$. Note that $\VERT B\VERT=\rho(B)$. 
As $Ax^\star-b-|x^\star|=0$, we have
\begin{align*}
&Ax-b-|x|=A(x-x^\star)-(|x|-|x^\star|)\\
& \Rightarrow x-x^\star=A^{-1}(Ax-b-|x|)+A^{-1}(|x|-|x^\star|)\\
& \Rightarrow |x-x^\star|\leq |A^{-1}||Ax-b-|x||+|A^{-1}|(||x|-|x^\star||)\\
& \Rightarrow (I-|A^{-1}|)|x-x^\star|\leq |A^{-1}||Ax-b-|x||
\end{align*}
By Neumann series theorem \cite{Horn}, $(I-|A^{-1}|)^{-1}$ and $(I-B)^{-1}$ exist and are non-negative. Hence,
\begin{align*}
& |x-x^\star|\leq (I-|A^{-1}|)^{-1} |A^{-1}| |Ax-b-|x||\\
& \Rightarrow  |x-x^\star|\leq (I-B)^{-1}B|Ax-b-|x||
\end{align*}
The last inequality follows from  $(I-|A^{-1}|)^{-1}=\sum_{i=0}^\infty |A^{-1}|^i\leq \sum_{i=0}^\infty B^i=(I-B)^{-1}$. Hence,
\begin{align*}
\VERT x-x^\star\VERT 
 &\leq \VERT(I-B)^{-1}\VERT\cdot \VERT B\VERT\cdot \VERT Ax-b-|x|\VERT\\
 &\leq \bigg(\sum_{i=0}^\infty\VERT B^i\VERT\bigg)\VERT B\VERT\cdot\VERT Ax-b-|x|\VERT\\
 &\leq \frac{\gamma}{1-\gamma}\VERT Ax-b-|x|\VERT.
\end{align*}
Moreover, for $d$ with $\|d\|_\infty \leq 1$, 
\begin{align*}
|(B^{-1}-\diag(d))^{-1}|
 = |(I-B\diag(d))^{-1}B|
 = \bigg| \sum_{i=0}^\infty (B\diag(d))^i B\bigg|
 \leq  \sum_{i=1}^\infty B^i.
\end{align*}
Since $\sum_{i=0}^\infty B^i=-(B^{-1}-I)^{-1}$, the Perron--Frobenius theorem implies $c_{\VERT\cdot\VERT}(B^{-1})=\frac{\gamma}{1-\gamma}$.
\end{proof} 

One may wonder why we do not use the well-known result which states the existence of  a matrix norm $\VERT. \VERT$ with $\VERT A \VERT<\rho$, see Lemma 5.6.10 in \cite{Horn}, to prove the above theorem. The underlying reason is that the given matrix norm by this result is not necessarily a scaled matrix $p$-norm. It is worth mentioning that, under the assumption of  Theorem \ref{JhJ}, when $|A^{-1}|>0$, one obtains
\begin{align}
c_{\VERT\cdot\VERT}(A)= \frac{\rho(|A^{-1}|)}{1-\rho(|A^{-1}|)},
\end{align}
for some scaling $1$-norm.
%%%%%%%% NEW
Note that a sufficient condition for having $\rho(|A^{-1}|)<1$ is the  existence of a diagonal matrix $S$ with $|S|=I$ such that  $A^{-1}S\geq 0$ and   $(A-S)^{-1}S\geq 0$. In fact, Theorem 5.2 in Chapter 7 of \cite{Berman} implies that  $\rho(A^{-1}S)<1$ under this condition, which is equivalent  to  $\rho(|A^{-1}|)<1$.

%%%%%%%%%%%%%%%%%%%%%%%%%%%%%%%%%%%%%%%%%
Error bound can be utilized as a tool in stability analysis \cite{Cottle, facchinei2007finite}. As mentioned earlier, \eqref{AVE}  has a unique solution for each $b\in\R^n$ if and only if $[A-I, A+I]$ is regular. We denote the set of matrices which satisfy this property by $\mathcal{A}$. It is easily seen that $\mathcal{A}$ is an open set. Let function $X(A, b)\colon  \mathcal{A}\times \mathbb{R}^n\to  \mathbb{R}^n$ return the solution of~\eqref{AVE}. In the following proposition, we list some properties of function $X$.

\begin{proposition}
Let $A\in\mathcal{A}$. 
\begin{enumerate}[i)]
\item
For any $b_1, b_2\in\mathbb{R}^n$,
\begin{align*}
\|X(A,b_1)-X(A,b_2)\|\leq c(A)\|b_1-b_2\|
\end{align*}
\item
Function $X$ is locally Lipschitz with modulus $c(A)$.
\end{enumerate}
\end{proposition}

\begin{proof}
First, we show the first part. Suppose that $X(A,b_1)=x_1$ and $X(A,b_2)=x_2$. Thus,
\begin{align*}
Ax_1-|x_1|-(Ax_2-|x_2|)=b_1-b_2.
\end{align*}
There exists a matrix $D\in [-I, I]$ such that $|x_2|-|x_1|= D(x_2-x_1)$. So the above equality can be written as 
\begin{align*}
(A+D)(x_1-x_2)=b_1-b_2,
\end{align*}
which implies that $\|x_1-x_2\|\leq \|(A+D)^{-1}\|\cdot\|b_1-b_2\|\leq c(A)\|b_1-b_2\|$.

Now, we prove the second part. Consider the locally Lipschitz function $\phi\colon  \mathcal{A}\times \mathbb{R}^n\times \mathbb{R}^n\to \mathbb{R}^n$ given by 
$\phi(A, b, x)=Ax-b-|x|$. We have $\partial_x\phi(A, b, x)\subseteq [A-I, A+I]$. As $[A-I, A+I]$ is regular, the implicit function theorem (see Chapter~7 in~\cite{Clarke}) implies that there exists  a locally Lipschitz function $X(A, b)\colon  \mathcal{A}\times \mathbb{R}^n\to  \mathbb{R}^n$  with $\phi(A, b, X(A, b))=0$. In addition,   
\begin{align*}
\|X(A_1,b_1)-X(A_2,b_2)\|\leq c(A)(\|A_1-A_2\|+\|b_1-b_2\|)
\end{align*}
where $A_1, A_2$ and $b_1, b_2$ are in some neighborhoods of $A$ and $b$, respectively.
\end{proof} 
%%%%%%%%%%%%%%%%%%%%%%%%%%%%%

As mentioned earlier, one class of effective approaches to handle \eqref{AVE} is concave optimization methods. Mangasarian \cite{MangasarianC1} proposed  the following concave optimization problem,
\begin{align}\label{quadr}
\min \ & e^T(Ax-b-| x|)\ \
 s.t. \ \ (A+I)x\geq b,\ 
  (A-I)x\geq b.
\end{align}
He showed that \eqref{AVE} has a solution if and only if the optimal value of \eqref{quadr}  is zero. Now, we show that \eqref{quadr} has weak sharp minima property. 
Consider an optimization problem $\min_{x\in X} f(x)$ with the optimal solution set~$S$. The set $S$ is called a weak sharp minima if there is an $\alpha>0$ such that
\begin{align*}
\alpha\cdot\dist_S(x)\leq f(x)-f(s), \quad \forall x\in X,\ \forall s\in S,
\end{align*}
where $\dist_S(x):=\min\{\|x -s\|_2 : s\in S\}$.
 Weak sharp minima notion has wide applications in the convergence analysis of iterative methods and error bounds \cite{Burke2, Burke}. 
\begin{proposition}
Let $A\in\mathcal{A}$. Then the optimal solution of \eqref{quadr} is a weak sharp minimum.
\end{proposition}

\begin{proof}
Let $X$ and $x^\star$ denote the feasible set and the unique solution of \eqref{quadr}, respectively. By Theorem \ref{TB}, $c_2(A)\in \mathbb{R}_+$ and 
\begin{align*}
\frac{1}{c_2(A)}\|x- x^\star\|_2 \leq \|Ax-b-|x| \|_2, \quad  \forall x\in X.
\end{align*}
As  $\|Ax-b-|x| \|_2\leq \|Ax-b-|x| \|_1$ and $Ax-b-|x| \geq 0$ for $x\in X$, we have
\begin{align*}
\frac{1}{c_2(A)}\|x- x^\star\|_2 \leq e^T(Ax-b-|x|), \quad  \forall x\in X,
\end{align*}
which shows that $x^\star$ is a weak sharp minimum.
\end{proof} 

%%%%%%%%%%%%%%%%%%%%%%%%%%%%%%%%%%%%%%%%%%%%%
\subsection{Condition number of AVE for 2-norm}

Since $\|A^{-1}\|_2=\frac{1}{\sigma_{\min}(A)}$, $c_2(A)$ can be computed as the optimal value of  the following optimization problem,
 \begin{align}\label{OSEr}
\min \ & \sigma_{\min}(A-\diag(d)) \ \ s.t. \ \ \|d\|_{\infty}\leq 1.
\end{align}
In general, the function $\sigma_{\min}(\cdot)$ is neither convex nor concave; see Remark 5.2 in \cite{Qi}. Here, 
$\sigma_{\min}(\cdot)$ is a function of diagonals. Nonetheless, $\sigma_{\min}(\cdot)$ is also neither convex nor concave in this case; the following example clarifies this point.
From this perspective, Proposition~\ref{Pr1} mentioned above is by far not obvious.

\begin{example}
Let $
A=\begin{pmatrix}
2 & 1\\
-2 & 1
\end{pmatrix}$.
We have
\begin{align*}
\sigma_2(A)=\sqrt{2}\leq \frac{1}{2}\sigma_{\min}(A+I)+\frac{1}{2}\sigma_{\min}(A-I)\approx 1.541,\\
\sigma_2(A)=\sqrt{2}\geq \frac{1}{2}\sigma_{\min}(A+E)+\frac{1}{2}\sigma_{\min}(A-E)\approx 1.34,
\end{align*}
where $E=\begin{pmatrix}
0 & 0\\
0 & 1
\end{pmatrix}$.
\end{example}

In the next proposition, we give a formula for symmetric matrices. Before we get to the proposition, we present a lemma. 
\begin{lemma}\label{kkkk}
Let $A$ be symmetric. The interval matrix $[A-I, A+I]$ is regular If and only if 
\begin{align}\label{fak}
  |\lambda_i(A)|>1,  \  \  i=1, ..., n
\end{align}
\end{lemma}

\begin{proof}
First, we show "if part". By spectral decomposition, $A$ can be written as $A=UDU^T$, where $U$ is an orthogonal matrix 
and $D=\diag(\lambda(A))$. Suppose that $\nu\in [-1, 1]$. We have $A+\nu I=U(D+\nu I)U^T$. Since $A+\nu I$ is invertible, for $ i=1, ..., n$, $ |\lambda_i(A)|>1$.

The "only if" part is resulted from Bauer-Fike Theorem \cite{Fike}.
\end{proof}

Note that condition \eqref{fak} is equivalent to $\sigma_{\min}(A)>1$.

\begin{proposition}
Let the interval matrix $[A-I, A+I]$ be regular. If $A$ is symmetric, then $c_2(A)=\frac{1}{\sigma_{\min}(A)-1}$.
\end{proposition}

\begin{proof}
As $[A-I, A+I]$ is regular, $\sigma_{\min}(A)>1$.  For $d$ with $\|d\|_\infty\leq 1$, 
$\sigma_{\min}(A+\diag(d))\geq \sigma_{\min}(A)-1$. 
By the proof of Lemma \ref{kkkk}, it is seen that there exists  $\bar d$ with $\|\bar d\|_\infty= 1$ shch that $\sigma_{\min}(A+\diag(\bar d))= \sigma_{\min}(A)-1$.
 Hence, the proposition follows from formulation~\eqref{OSEr}.
\end{proof} 

\begin{proposition}\label{Prpp}
If $\sigma_{\min}(A)>1$, then
\begin{align}\label{ineqPropcSn}
c_2(A)\leq \frac{1}{\sigma_{\min}(A)-1}.
\end{align}
\end{proposition}

\begin{proof}
Note that under the assumption, \eqref{AVE} has a unique solution for any $b$, see Proposition 3 in \cite{mangasarian2}, and consequently $[A-I, A+I]$ is regular. Let
 $\hat d\in\{d:\|d\|_\infty\leq 1\}$. Consider the formulation \eqref{OSEr}. Since,
$\sigma_{\min}(A+B)\geq \sigma_{\min}(A)-\|B\|_2$ and $\max_{\|d\|_\infty\leq 1} \|\diag(d)\|_2=1,$ 
we obtain the desired inequality.
\end{proof} 

In the following example we show that the bound \eqref{ineqPropcSn} can be arbitrary large while  the error bound with respect to $2$-norm is bounded. 

\begin{example}
Let $\epsilon >0$ and
$$A=\frac{\sqrt{2}}{2}\begin{pmatrix}
1 & -1\\
1 & 1
\end{pmatrix}
\begin{pmatrix}
5 & 0\\
0 & 1+\epsilon
\end{pmatrix}.$$
As $\sigma_{\min}(A)=1+\epsilon$, we have the assumption of Proposition~\ref{Prpp}. By Proposition~\ref{Pr1},
\begin{align*}
c_2(A)=\max
 \left\{\|(A-I)^{-1}\|_2, \|(A-E)^{-1}\|_2, \|(A+E)^{-1}\|_2, \|(A+I)^{-1}\|_2\right\},
\end{align*}
where $E=\begin{pmatrix}
-1 & 0\\
0 & 1
\end{pmatrix}$. With a little algebra, it is seen that $ c_2(A)\leq 6$, while $\frac{1}{\sigma_{\min}(A)-1}$ goes to infinity as $\epsilon$ tends to zero.
\end{example}

 For matrix $A$, let 
\begin{align*}
r_i(A)=\sum_{\substack {j=1\\ j\neq i}}^n |A_{ij}|, \ \ \cl_i(A)=\sum_{\substack {j=1\\ j\neq i}}^n |A_{ji}|.
\end{align*}

\begin{proposition}\label{GFGF} 
Let $\bar d=\sgn(\Diag(A))$. If 
\begin{align*}
\alpha:=\min_{i=1, ..., n}\{|A_{ii}|-\textstyle\frac{1}{2}(r_i(A)+\cl_i(A))\}>1, 
\end{align*}
then $c_2(A)= \| (A-\diag(\bar d))^{-1}\|_2$. 
\end{proposition}

\begin{proof}
Let $d\in\{d:\|d\|_\infty\leq 1\}$. By Theorem 3 in \cite{johnson}, $\sigma_{\min}(A-\diag(d))\geq \alpha-1$. So $[A-I, A+I]$ is regular. 
Since $\|A^{-1}\|^{-2}_2 = \lambda_{\min}(A^TA)$, by Proposition~\ref{Pr1},
$c_2(A)^{-2}=\min_{|d|=e}\lambda_{\min}\big((A-\diag(d))^T(A-\diag(d))\big)$. Suppose that $|d|=e$. Consider matrix  
\begin{align*}
T&=(A-\diag(d))^T(A-\diag(d))-(A-\diag(\bar d))^T(A-\diag(\bar d))\\
&=\diag(\bar d)A+A^T\diag(\bar d)-\diag( d)A-A^T\diag( d).
\end{align*}
It is easily seen that $T$ is diagonally dominant with nonnegative diagonal, so it is positive semi-definite. Consequently, 
$\lambda_{\min}\big((A-\diag(d))^T(A-\diag(d))\big)\geq \lambda_{\min}\big((A-\diag(\bar d))^T(A-\diag(\bar d))\big)$,
which implies the desired equality.
\end{proof} 

Note that under the assumptions of Proposition \ref{GFGF}, we also have the following bound
\begin{align*}
c_2(A)\leq \frac{1}{\alpha-1}.
\end{align*}
As for a permutation matrix $P$, $\|AP\|_2=\|A\|_2$ and $[-I, I]P=[-I, I]$, the following corollary gives a more generalized form of Proposition \ref{GFGF}. 

\begin{corollary}\label{GFGFGF} 
Let $P$ be a permutation matrix and $B=AP$.  If 
\begin{align*}
\alpha:=\min_{i=1, ..., n}\{|B_{ii}|-\textstyle\frac{1}{2}(r_i(B)+\cl_i(B))\}>1, 
\end{align*}
then $c_2(A)= \| (B-\diag(\bar d))^{-1}\|_2$, where  $\bar d=\sgn(\Diag(B))$.
\end{corollary}

%%%%%%%%%%%%%%%%%%%%%%%%%%%%%%%%%%%%%%%%%%%%%
\subsection{Condition number of AVE for $\infty$-norm}

Some upper bounds were proposed for $\|A^{-1}\|_\infty$ and $\|A^{-1}\|_1$; see \cite{kolotilina, Li, moravca, varah}. As Theorem \ref{TB} holds for any scaling $p$-norm, it would be advantageous to use these norms.

\begin{proposition}\label{PP}
If $(A-I)^{-1}\geq 0$ and $(A+I)^{-1}\geq 0$, 
then $c_\infty(A)=\|(A-I)^{-1}e\|_\infty$.
\end{proposition}

\begin{proof}
By Kuttler's theorem \cite{Kuttler}, under the assumptions of the proposition, the interval matrix $[A-I, A+I]$ is regular and inverse nonnegative. In addition,  $[A-I, A+I]^{-1}=[(A+I)^{-1}, (A-I)^{-1}]$. It is easily seen that for any non-negative matrix $M$ we have $\|M\|_\infty=\|Me\|_\infty$. Hence, $c_\infty(A)=\|(A-I)^{-1}e\|_\infty$.
\end{proof} 

%%%%%%%%%%%%%%%%%%%%%%%%%%%%%%%%%%%%%%%%%%%%%%%%%%
\begin{proposition}\label{BAISR}
If $\rho(|A^{-1}|)<1$, then 
\begin{align}
c_\infty(A)\leq \|\max(|B_1|, |B_2|)\|_\infty.
\end{align}
where $H=(I-|A^{-1}|)^{-1}$, $T=(2\diag(\Diag(H))-I)^{-1}$ and 
\begin{align*}
& B_1=\min\{-H|A^{-1}|+T(A^{-1}+|A^{-1}|), T(-H|A^{-1}|+T(A^{-1}+|A^{-1}|))\},\\
& B_2=\max\{H|A^{-1}|+T(A^{-1}-|A^{-1}|), T(H|A^{-1}|+T(A^{-1}-|A^{-1}|))\}
\end{align*} 
\end{proposition}

\begin{proof}
By Theorem 2.40 in \cite{Rohn2}, $[A-I, A+I]^{-1}\subseteq [B_1, B_2]$. Thus,
\begin{align*}
c_\infty(A)
 = \max_{\|d\|_{\infty}\leq 1} \|(A-\diag(d))^{-1}\|_\infty
 \leq \max_{X\in [B_1, B_2]} \|X\|_\infty\leq \|\max(|B_1|, |B_2|)\|_\infty.
\end{align*}
\end{proof} 

%%%%%%%%%%%%%%%%%%%%%%%%%%%%%%%%%%%%%%%%%%%%%%
\begin{proposition}
Let $A$ be an M-matrix. If $\rho(A^{-1})<1$, then 
\begin{align}
c_{\infty}(A)= \| (A- I)^{-1}e\|_{\infty}.
\end{align}
\end{proposition}

\begin{proof}
By virtue of  Theorem 3.6.3  in \cite{Arnold}, $A-I$ is an M-matrix. In addition,  as  M-matrices are preserved by the addition of positive diagonal matrices \cite{Berman}, $A+I$ is also an M-matrix. Hence, by Kuttler's theorem \cite{Kuttler},  $[A-I,   A+I]$ is 
inverse nonnegative, and we proceed as in the proof of Proposition~\ref{PP}.
%regular and 
%$[A-I,   A+I]^{-1}=[  (A+ I)^{-1},  (A- I)^{-1} ]$. Therefore,
%\begin{align*}
%c_{\infty}(A)=&\max_{\|d\|_\infty\leq 1} \ \|(A-\diag(d))^{-1}\|_{\infty}
%=\max_{X\in [A-I,   A+I]^{-1}} \ \|X\|_{\infty}
% \\&= \|(A- I)^{-1}e\|_{\infty}.
%\end{align*}
\end{proof} 

%%%%%%%%%%%%%%%%%%%%%%%%%%%%%%%%%%%%%%%%%%%%%%%
\begin{proposition}
Let $A$ be an H-matrix. If $\rho(\langle A\rangle ^{-1})<1$, then 
\begin{align}
c_{\infty}(A)\leq \|(\langle  A\rangle-I)^{-1}e\|_{\infty}.
\end{align}
\end{proposition}

\begin{proof}
Under the assumption, Theorem 3.7.5 in \cite{Arnold} implies that interval matrix $[A-I, A+I]$ is an H-matrix. So, $[A-I, A+I]$ is regular. In addition,  
$\langle [A-I, A+I]\rangle=[\langle  A\rangle-I,  \langle  A\rangle+I]$. 
By Kuttler's theorem \cite{Kuttler},  
$[\langle  A\rangle-I,  \langle  A\rangle+I]^{-1}=[(\langle  A\rangle+I)^{-1},  (\langle  A\rangle-I)^{-1}]$. Because $(\langle  A\rangle+I)^{-1}\geq 0$,
\begin{align*}
c_{\infty}(A)=\max_{\|d\|_\infty\leq 1} \ \|(A-\diag(d))^{-1}\|_{\infty} \\\leq \max_{\|d\|_\infty\leq 1} \ \|\langle A-\diag(d)\rangle ^{-1}\|_{\infty} \\= \|(\langle  A\rangle-I)^{-1}e\|_{\infty},
\end{align*}
where the first inequality follows form the fact that for H-matrix $A$, $\|A^{-1}\|_{\infty}\leq \|\langle  A\rangle^{-1}\|_{\infty}$; see Theorem 1 in \cite{Vagra}.
\end{proof} 

\begin{proposition}
Let $r>0$ and $\VERT x \VERT:=\| \diag(r)^{-1}x\|_\infty$. If 
\begin{align*}
\alpha:=\min_{i=1, ..., n}\{|A_{ii}|-1-r_i^{-1}\sum_{j\neq i} r_j|A_{ij}|\}>0, 
\end{align*}
then $c_{\VERT\cdot\VERT}(A)\leq \frac{1}{\alpha}$. 
\end{proposition}

\begin{proof}
First, we show that for a given $d$ with $\| d \|_\infty\leq 1$, we have the following inequality 
\begin{align}\label{gggg}
\min_{\VERT x \VERT=1}\VERT (A+\diag(d))x \VERT\geq \alpha.
\end{align}
Suppose that $\bar x\in\argmin_{\VERT x \VERT=1}\VERT (A+\diag(d))x \VERT$ and $ \bar x_k=r_k$. We have 
\begin{align*}
\VERT (A+\diag(d))\bar x \VERT& \geq |r_k^{-1}(A-\diag(d))_{k*}\bar x|\\
&\geq  |A_{kk}-d_k|-r_k^{-1}\sum_{j\neq k} |A_{kj}\bar x_j|\\\
&\geq |A_{kk}|-1-r_k^{-1}\sum_{j\neq k} \frac{r_j|\bar x_j|}{r_j}|A_{kj}|\\
&\geq |A_{kk}|-1-r_k^{-1}\sum_{j\neq k} |A_{kj}|r_j\geq \alpha.
\end{align*}
Consequently, interval matrix $[A-I, A+I]$ is regular. Similarly to the proof of Proposition~\ref{GFGF}, one can show that 
\begin{align*}
c_{\VERT\cdot\VERT}(A)^{-1}=\min_{\| d \|_\infty\leq 1, \VERT x \VERT=1}\VERT (A+\diag(d))x \VERT.
\end{align*}
The above equality and \eqref{gggg} imply $c_{\VERT\cdot\VERT}(A)\leq \frac{1}{\alpha}$, and the proof is complete. 
\end{proof}

\begin{corollary} 
If 
$
\alpha:=\min_{i=1, ..., n}\{|A_{ii}|-r_i(A)\}>1, 
$
then $c_{\infty}(A)\leq \frac{1}{\alpha-1}$. 
\end{corollary}
\begin{corollary} 
If 
$
\beta:=\min_{j=1, ..., n}\{|A_{jj}|-\cl_j(A)\}>1, 
$
then $c_{1}(A)\leq \frac{1}{\beta-1}$. 
\end{corollary}
%%%%%%%%%%%%%%%%%%%%%%%%%%%%%%%%%%%%%%%%%%%%%
\section{Error bounds and a condition number of AVE related to linear complementarity problems}\label{sLCP}

The study of  AVE is inspired from the well-known linear complementarity problem (LCP) \cite{mangasarian2}. LCP provides a unified framework for many mathematical programs \cite{Cottle}.
In the section, we study error bounds for AVE obtained by transforming LCPs. Consider a general linear complementarity problem
\begin{align}\label{LCP}\tag{LCP}
Mx+q\geq 0, \ \  x\geq 0,  \ \  x^T(Mx+q)=0,
\end{align}
where $M\in\mathbb{R}^{n\times n}$ and $q\in\mathbb{R}^{n}$. Throughout the section, without loss of generality, we may assume that one is not an eigenvalue of $M$. So matrix $(M-I)$ is non-singular. This assumption is not restrictive, as one can rescale $M$ and $q$ in \eqref{LCP}. Problem \eqref{LCP} can be formulated as the following AVE,
\begin{align}\label{LCP-AVE}
(M+I)(M-I)^{-1}(x+q)=|x|;
\end{align}
see \cite{Mang}. The following proposition states the relationship between $M$ and $(M+I)(M-I)^{-1}$; see Theorem~2 in \cite{Rohnrump}.

\begin{proposition}
Let $M-I$ be non-singular. Matrix $M$ is a P-matrix if and only if  $[(M+I)(M-I)^{-1}-I, (M+I)(M-I)^{-1}+I]$ is regular. 
\end{proposition}
 
In addition to the error bounds introduced for some classes of matrices in the former section, in the following  results, we propose error bounds for absolute value equation \eqref{LCP-AVE} according to some properties of $M$. 

\begin{proposition}\label{PLC}
Let $M$ be an M-matrix with $\Diag(M)\leq e$ and $M-I$ is nonsingular. Then 
\begin{align*}
c((M+I)(M-I)^{-1})=\frac{1}{2}\|I-M^{-1}\|.
\end{align*}
\end{proposition}

\begin{proof}
Since the off-diagonal elements of $M$ are non-positive and $M^{-1}\geq 0$, we have $\Diag(M^{-1})\geq e$. 
Putting $A=(M+I)(M-I)^{-1}$, we get
\begin{align*}
& A-I=((M+I)-(M-I))(M-I)^{-1}=2(M-I)^{-1},\\
& A+I=2M(M-I)^{-1}=2(I-M^{-1})^{-1}.
\end{align*}
Therefore, $(A-I)^{-1}=\frac{1}{2}(M-I)\leq 0$ and $(A+I)^{-1}=\frac{1}{2}(I-M^{-1})\leq 0$. Kuttler's theorem \cite{Kuttler} implies that $[A-I, A+I]$ is regular and 
 $[A-I, A+I]^{-1}\subseteq\frac{1}{2}[I-M^{-1},M-I]$, and consequently, $c(A)=\frac{1}{2}\|I-M^{-1}\|$.
\end{proof} 

It is worth noting that the assumption $\Diag(M)\leq e$ is not restrictive, since LCP$(M, q)$ is equivalent to LCP$(\lambda M, \lambda q)$ for $\lambda>0$. 
In the following, we investigate the case that $M$ is an H-matrix. Before we get to the theorem, which gives a bound in this case, we need to present a lemma first.

\begin{lemma}\label{LEm1}
If $M$ is an H-matrix with non-negative diagonals, then $M+I$ is an H-matrix.
\end{lemma}

\begin{proof}
By $I_{27}$ of Theorem 2.3 in Chapter 6 of \cite{Berman}, there exist $x>0$ such that $\langle  M\rangle x>0$, which implies $(\langle  M\rangle+I)x>0$. 
By the assumption, $\langle  M+I\rangle=\langle  M\rangle+I$. By using the aforementioned theorem,  $M+I$ should be an H-matrix.
\end{proof}

\begin{theorem}\label{T33}
Let $M-I$ be nonsingular and let $M$ be an H-matrix with $0\leq \Diag(M)\leq e$. Then
\begin{align*}
c((M+I)(M-I)^{-1})\leq\frac{1}{2}\|\langle  M\rangle^{-1}-I\|.
\end{align*}
\end{theorem}

\begin{proof}
Consider vector $d\in\mathbb{R}^n$ with $\|d\|_\infty\leq 1$. We have
\begin{align*}
|(M-I)(M+I)^{-1}\diag(d)|&\leq  |(M-I)(M+I)^{-1}|\\
& \leq   |M-I| |(M+I)^{-1}|\\
& \leq  (I-\langle  M\rangle)(\langle  M\rangle+I)^{-1} 
\end{align*}
where the last inequality follows from  $|M-I|\leq I-\langle  M\rangle$,  $ |(M+I)^{-1}|\leq (\langle  M+I\rangle)^{-1}$; see Theorem~3.7.5 in \cite{Arnold} and  Lemma~\ref{LEm1}. Thus, $\rho((M-I)(M+I)^{-1}\diag(d))\leq \rho((I-\langle  M\rangle)(\langle  M\rangle+I)^{-1})$. Since $\langle  M\rangle$ is an M-matrix and $\rho(BC)=\rho(CB)$, we have $\rho((I-\langle  M\rangle)(\langle  M\rangle+I)^{-1})<1$; see $G_{22}$ of Theorem 2.3 in Chapter 6 of \cite{Berman}. Hence, $\rho((M-I)(M+I)^{-1}\diag(d))<1$.

Let $\hat A\in [(M+I)(M-I)^{-1}-I, (M+I)(M-I)^{-1}+I]$. So $\hat A=(M+I)(M-I)^{-1}-\diag(d)$ for some $d$ with $\|d\|_\infty\leq 1$. Hence,
\begin{multline*}
((M+I)(M-I)^{-1}-\diag(d))^{-1}\\
=(I-(M-I)(M+I)^{-1}\diag(d))^{-1}(M-I)(M+I)^{-1}.
\end{multline*}
By applying Neumann series and the obtained results, we have 
\begin{align*}
|((M+I)(M-I)^{-1}&-\diag(d))^{-1}|\\
 &\leq \left|\sum_{i=0}^\infty ((M-I)(M+I)^{-1}\diag(d))^i\right|
   |(M-I)(M+I)^{-1}|,\\
&\leq \sum_{i=1}^\infty |(M-I)(M+I)^{-1}|^i ,\\
&\leq \sum_{i=1}^\infty ((I-\langle  M\rangle)(\langle  M\rangle+I)^{-1})^i,\\
&= (I-(I-\langle  M\rangle)(\langle  M\rangle+I)^{-1})^{-1}(I-\langle  M\rangle)(\langle  M\rangle+I)^{-1}\\
&= \frac{1}{2}(\langle  M\rangle^{-1}-I),
\end{align*}
where the last equality obtained by using the relations $(I-A)^{-1}A=(A^{-1}-I)^{-1}$ and $((I+\langle  M\rangle)(I-\langle  M\rangle)^{-1}-I)^{-1}=(2\langle  M\rangle(I-\langle  M\rangle)^{-1})^{-1}=\frac{1}{2}(\langle  M\rangle^{-1}-I)$. 
Therefore, $\|\hat A^{-1}\|\leq \frac{1}{2}\|\langle  M\rangle^{-1}-I\|$, and the proof is complete.
\end{proof} 

In the rest of this section, by using the obtained result, we present new error bounds for linear complementarity problems. Many papers have devoted to the  error bounds for the LCP(M, q); see \cite{Chen, Chen2, Cottle, Garcia, Pang}. It is easily seen that $\hat x$ is a solution of \eqref{LCP} if and only if $\hat x$ solves 
\begin{align*}
\theta(x):=\min(Mx+q, x)=0.
\end{align*}
The function $\theta(x)$ is called  the natural residual of \eqref{LCP}. As mentioned earlier, \eqref{LCP} has a unique solution for each $q$ if and only if $M$ is a P-matrix. For $M$ being a P-matrix,  Chen and Xiang~\cite{Chen} proposed the following error bound
\begin{align*}
\|x-x^\star\|\leq \max_{0\leq D\leq I} \|(I-D+DM)^{-1}\|\cdot \|\theta(x)\|,
\end{align*}
where $x^\star$ is the solution of \eqref{LCP} and $x\in\R^n$ arbitrary.  
%%%%%%%%%%%%%%%%%%%%%%%%%%%%%%%%%%%%%%%%%%%%% Modified
By introducing new variable $d$ with $\diag(d)=2D-I$, we have 
\begin{align}\label{Chh}
\nonumber \max_{0\leq D\leq I} \|(I-D+& DM)^{-1}\|
 =\max_{\|d\|_\infty\leq 1}
   \|(I-\textstyle\frac{1}{2}(\diag(d)+I)+\textstyle\frac{1}{2}(\diag(d)+I)M)^{-1}\|\\
&= 2\max_{\|d\|_\infty\leq 1} \|(I-M)^{-1}((I+M)(I-M)^{-1}-\diag(d))^{-1}\|.
\end{align}
Because $c(A)=c(-A)$, we have 
\begin{align*}
\max_{0\leq D\leq I} \|(I-D+DM)^{-1}\|\leq 2c((I+M)(M-I)^{-1})\|(I-M)^{-1}\|.
\end{align*}
Therefore, the given results in this paper can be exploited for providing an upper bound for this maximization. For instance, Chen and Xiang, see Theorem~2.2 in~\cite{Chen}, proved that when $M$ is an M-matrix, then
\begin{align*}
\max_{0\leq D\leq I}\, \|(I-D+DM)^{-1}\|_1=\max_{v\in V}\, f(v),
\end{align*}
where $f(v)=\max_{1\leq i \leq n}(e+v-M^Tv)_i$ and $V=\{v: M^Tv\leq e, v\geq 0\}$. As seen, $f$ is a piece-wise linear convex function. However, maximization of a convex function, in general, is an intractable problem. In this case, one needs to solve $n$ linear programs. In the next proposition, we give an upper bound for the optimal value for $\infty$-norm. 

\begin{proposition}
Let $M$ be an M-matrix with $\Diag(M)\leq e$. Then
\begin{align*}
\max_{0\leq D\leq I} \|(I-D+DM)^{-1}\|_\infty = \|\hat{B}\|_\infty,
\end{align*}
where for $i, j=1, ..., n$
\begin{align*}
\unum{B}_{ij}=\sum_{k=1}^n \min\{(I-M)^{-1}_{ik}(I-M)_{kj}, (I-M)^{-1}_{ik}(M^{-1}-I)_{kj} \},\\
\onum{B}_{ij}=\sum_{k=1}^n \max\{(I-M)^{-1}_{ik}(I-M)_{kj}, (I-M)^{-1}_{ik}(M^{-1}-I)_{kj} \},
\end{align*} 
and $\hat{B}=\max(| \unum{B}|, | \onum{B}|)$.
\end{proposition}

\begin{proof}
Similarly to the proof of Proposition \ref{PLC}, if $\Diag(M)\leq e$, we have $[(I+M)(I-M)^{-1}-I, (I+M)(I-M)^{-1}+I]^{-1}=\frac{1}{2}[I-M, M^{-1}-I]$. Therefore, by $\eqref{Chh}$
\begin{align*}
\max_{0\leq D\leq I}\, \|(I-D+DM)^{-1}\|_\infty
\leq \max_{I-M\leq X \leq M^{-1}-I} \|(I-M)^{-1} X\|_\infty.
\end{align*}
Furthermore, $\{ (I-M)^{-1} X:  I-M\leq X \leq M^{-1}-I \}\subseteq [\unum{B}, \onum{B}]$. 
Hence, 
\begin{align*}
\max_{0\leq D\leq I} \|(I-D+DM)^{-1}\|_\infty\leq \|\hat{B}\|_\infty.
\end{align*}
On the other hand, suppose that $\|\hat{B}\|_\infty=\|\hat{B}_{i\ast}\|_\infty$. There exist $\check{B}\in \{ (I-M)^{-1} X:  I-M\leq X \leq M^{-1}-I \}$ such that
 $| \check{B}_{i\ast}|=\hat{B}_{i\ast}$, which implies the above inequality holds as equality, and the proof will be complete. 
\end{proof} 

For $M$ being an H-matrix with $0\leq \Diag(M)\leq e$, similarly to the proof of Theorem~\ref{T33}, one can show that for $d$ with $\|d\|_\infty\leq 1$,
\begin{align*}
|(I-M)^{-1}((I+M)&(I-M)^{-1}-\diag(d))^{-1}|\\
&= |((I+M)-\diag(d)(I-M))^{-1}|,\\
&=\left|(I+M)^{-1}\sum_{i=0}^\infty (\diag(d)(I-M)(M+I)^{-1})^i\right|,\\
&\leq (\langle  M\rangle+I)^{-1}
 \sum_{i=0}^\infty ((I-\langle  M\rangle)(\langle  M\rangle+I)^{-1})^i,\\
&= (\langle  M\rangle+I)^{-1}(I-(I-\langle  M\rangle)(\langle  M\rangle+I)^{-1})^{-1}
 =\frac{1}{2}\langle  M\rangle^{-1}.
\end{align*}
Therefore, by $\eqref{Chh}$, we get 
\begin{align}\label{CHen}
\max_{0\leq D\leq I} \|(I-D+DM)^{-1}\|
\leq \langle  M\rangle^{-1},
\end{align}
 which is a well-known bound; see Theorem 2.1 in \cite{Chen}. Here, we obtain inequlity \eqref{CHen} with a different method as a by-product of our analysis.

%%%%%%%%%%%%%%%%%%%%%%%%%%%%
\section{Relative condition number of AVE}\label{S.RC}

We introduce a relative condition number as follows
\begin{align*}
c^*(A)
:=\max_{\|d\|_\infty\leq 1}\|(A-\diag(d))^{-1}\| \cdot \max_{\|d\|_\infty\leq 1}\|A-\diag(d)\|,
\end{align*}
which is equal to $c(A) \max_{\|d\|_\infty\leq 1}\|A-\diag(d)\|$. The meaning of the relative condition number follows from the bounds presented in the proposition below. They extend the  bounds known for the error of standard linear systems of equations~\cite{higham}.

\begin{proposition}
 If the interval matrix $[A-I, A+I]$ is regular and $b\neq 0$, then for each $x\in \mathbb{R}^n$
\begin{align*}
 c^*(A)^{-1} \frac{\|Ax-b-|x|\|}{\|b\|}
  \leq \frac{\|x- x^\star\|}{\|x^\star\|}
  \leq c^*(A) \frac{\|Ax-b-|x|\|}{\|b\|}.
\end{align*}
\end{proposition}

\begin{proof}
Since $b\neq 0$, we have $x^\star\neq 0$. First, we show the upper bound.
Denote $s^\star:=\sgn(x^\star)$, From $Ax^\star-b=|x^\star|=\diag(s^\star)x^\star$  we derive $(A-\diag(s^\star))x^\star=b$, from which $\|A-\diag(s^\star)\|\cdot\|x^\star\|\geq\|b\|$. Now, we have by Theorem~\ref{TB}
\begin{align*}
\|x- x^\star\|
 \leq c(A)\|Ax-b-|x|\|
 \leq c(A) \|Ax-b-|x|\|\frac{\|A-\diag(s^\star)\|\cdot\|x^\star\|}{\|b\|},
\end{align*}
from which the bound follows.

Now, we establish the lower bound.
From the proof of Theorem~\ref{TB} we know that there exist some $\hat A\in [A-I, A+I]$ such that $Ax-b-|x|=\hat{A}(x-x^\star)$. Hence
\begin{align*}
\|Ax-b-|x|\|
&= \|\hat{A}(x-x^\star)\|
 \leq \|\hat{A}\|\cdot\|x-x^\star\|\\
&\leq \|\hat{A}\|\cdot\|x-x^\star\|\frac{\|(A-\diag(s^\star))^{-1}\|\cdot\|b\|}{\|x^\star\|},
\end{align*}
from which the statement follows.
\end{proof} 

In order to compute $c^*(A)$ we have to determine $c(A)$ and $\max_{\|d\|_\infty\leq 1}\|A-\diag(d)\|$. The former is discussed in detail in the previous sections, so we focus on the latter now. Recall that a norm is absolute if $\|A\|=\||A|\|$, and it is monotone if $|A|\leq|B|$ implies $\|A\|\leq\|B\|$. For example, 1-norm, $\infty$-norm, Frobenius norm or max norm are both absolute and monotone.

\begin{proposition}
For any absolute and monotone matrix norm 
\begin{align*}
 \max_{\|d\|_\infty\leq 1}\|A-\diag(d)\|=\||A|+I_n\|.
\end{align*}
\end{proposition}

\begin{proof}
We have 
\begin{align*}
\max_{\|d\|_\infty\leq 1}\|A-\diag(d)\|
\leq \max_{\|d\|_\infty\leq 1}\||A|+|\diag(d)|\|
=\||A|+I\|,
\end{align*}
and equation is attained for certain $d$ with $\|d\|_\infty=1$.
\end{proof}

\begin{proposition}
For spectral norm we have
\begin{align*}
 \max_{\|d\|_\infty\leq 1}\|A-\diag(d)\|_2 \leq \|A\|_2+1.
\end{align*}
Moreover, It holds as an equality when $A$ is symmetric.
\end{proposition}

\begin{proof}
We have 
$\|A-\diag(d)\|_2\leq \|A\|_2+\|\diag(d)\|_2\leq \|A\|_2+1$.
For symmetric $A$, equality is attained for some $d$ with $\|d\|_\infty=1$.
\end{proof}

 \textbf{Conclusion. }In this paper, we studied error bounds for absolute value equations. We suggested some formulas for the computation of error bounds for some classes of matrices. The investigation of  other classes of matrices may be of interest for further research. The proposed formulas can  be employed not only for absolute value equations obtained by transforming  linear complementarity problems, but also for linear complementarity problems. In addition, We showed that , in general, the computation of error bounds, except for 2-norm,  for a general matrix is an NP-hard problem, and it remains an open problem for 2-norm.  

%%%%%%%%%%%%%%%%%%%%%%%%%%%%%%%
\section*{Acknowledgments}
%We would like to acknowledge the assistance of volunteers in putting
%together this example manuscript and supplement.
The authors were supported by the Czech Science Foundation Grant P403-18-04735S.

\bibliographystyle{siamplain}
\bibliography{references}
\end{document}